\documentclass[oneside]{amsart}
\usepackage{enumerate,amssymb,mathtools}
\usepackage{mathrsfs}
\newtheorem{thm}{Theorem}[section]
\newtheorem{coro}[thm]{Corollary}
\newtheorem{prop}[thm]{Proposition}
\newtheorem{lem}[thm]{Lemma}
\newtheorem*{thmun}{Theorem}
\theoremstyle{definition}
\newtheorem{defn}[thm]{Definition}
\newtheorem{ex}[thm]{Example}
\newtheorem{rem}[thm]{Remark}
\newtheorem{question}[thm]{Question}
\newcommand{\Rset}{\mathbb{R}}
\newcommand{\nset}{\omega}
\newcommand{\abs}[1]{\lvert#1\rvert}
\newcommand{\card}[1]{\lvert#1\rvert}
\newcommand{\eps}{\varepsilon}
\newcommand{\del}{\delta}
\newcommand{\Del}{\Delta}
\newcommand{\dd}{\,\mathrm{d}}

\newcommand{\zz}{\mathsf{Z}}
\newcommand{\TT}{\mathsf{T}}
\newcommand{\ttt}{\mathsf{T}^{\bullet}}
\newcommand{\yy}{\mathsf{K}}
\newcommand{\teta}{\theta}
\newcommand{\ttau}{{\boldsymbol\tau}}
\newcommand{\clue}[2]{\overset{\eqref{#1}}{#2}}
\newcommand{\haus}[1]{\mathscr{H}^{#1}}
\newcommand{\pack}[1]{\mathscr{P}^{#1}}
\newcommand{\upack}[1]{\overline{\mathscr{P}}^{#1}}
\newcommand{\lpack}[1]{\underline{\mathscr{P}}^{#1}}
\newcommand{\dpack}[1]{\underrightarrow{\mathscr{P}}^{#1}}

\newcommand{\uPack}[1]{\overline{\mathsf{P}}^{#1}}
\newcommand{\lPack}[1]{\underline{\mathsf{P}}^{#1}}
\newcommand{\dPack}[1]{\underrightarrow{\mathsf{P}}^{#1}}
\newcommand{\nuu}{\boldsymbol{\nu}}

\newcommand{\boxm}[1]{\nuu^{#1}}
\newcommand{\lboxm}[1]{\underline{\nuu}^{#1}}
\newcommand{\dboxm}[1]{\underrightarrow{\nuu}^{#1}}
\newcommand{\uboxm}[1]{\overline{\nuu}{}^{#1}}
\DeclareMathOperator{\aDim}{aDim}
\DeclareMathOperator{\updim}{\overline{dim}_{\mathsf{P}}}
\DeclareMathOperator{\hdim}{dim_{\mathsf{H}}}
\DeclareMathOperator{\adim}{dim_{\mathsf{A}}}
\DeclareMathOperator{\sadim}{dim_{\sigma\mathsf{A}}}
\DeclareMathOperator{\lpdim}{\underline{dim}_{\mathsf{P}}}
\DeclareMathOperator{\ubdim}{\overline{dim}_{\mathsf{B}}}
\DeclareMathOperator{\lbdim}{\underline{dim}_{\mathsf{B}}}
\DeclareMathOperator{\dpdim}{\underrightarrow{\mathrm{dim}}_{\mathsf{P}}}
\DeclareMathOperator{\gap}{\mathsf{gap}}
\DeclareMathOperator{\dist}{dist}
\DeclareMathOperator{\diam}{diam}
\newcommand{\arrow}[1]{\overrightarrow{#1}}
\newcommand{\upint}[1][]{\int_{#1}^*\hspace{-1ex}}
\newcommand{\clos}[1]{\overline{#1}}
\newcommand{\upto}{{\nearrow}}
\newcommand{\subs}{\subseteq}
\renewcommand{\leq}{\leqslant}
\renewcommand{\geq}{\geqslant}
\newcommand{\tree}{2^{<\omega}}
\newcommand{\ctree}{2^{\omega}}
\newcommand{\seq}[1]{\langle#1\rangle}
\newenvironment{enum}{\begin{enumerate}[\rm(i)]}{\end{enumerate}}

\begin{document}

\title
[Packing measures and dimensions on cartesian products]
{Packing measures and dimensions\\ on cartesian products}

\author{Ond\v rej Zindulka}
\address
{Department of Mathematics\\
Faculty of Civil Engineering\\
Czech Technical University\\
Th\'akurova 7\\
160 00 Prague 6\\
Czech Republic}
\email{zindulka@mat.fsv.cvut.cz}
\urladdr{http://mat.fsv.cvut.cz/zindulka}
\thanks{This project was supported by Department of Education
of the Czech Republic, research project BA MSM 210000010}

\begin{abstract}
Packing measures $\pack{g}(E)$ and Hewitt-Stromberg measures
$\boxm{g}(E)$
and their relatives are investigated.
It is shown, for instance, that for any metric spaces $X,Y$
and any Hausdorff functions $f,g$
$$
    \boxm{g}(X)\cdot\pack{h}(Y)\leq\pack{gh}(X\times Y)
$$
The inequality for the corresponding dimensions is established and
used for a solution of a problem of Hu and Taylor:
If $X\subs\Rset^n$, then
$$
    \inf\{\updim X\times Y-\updim Y:Y\subs\Rset^n\}
    =\liminf_{X_n\upto X}\lbdim X_n.
$$
Corresponding dimension inequalities for products of measures
are established.
%
\end{abstract}

\maketitle
\section{Introduction}

Consider separable metric spaces and their Hausdorff, packing
and lower packing dimensions denoted, respectively, by $\hdim$, $\updim$ and
$\lpdim$ (the definitions are provided below).
In 1981 Tricot~\cite{MR633256} proved that if $X,Y\subs\Rset^n$, then
\begin{equation}\label{eq:1}
  \hdim X+\updim Y\leq\updim X\times Y
\end{equation}
and this inequality was later generalized to arbitrary separable metric spaces
by Howroyd~\cite{MR1362951}.
In 1993 Hu and Taylor~\cite[(3.12)]{MR1376540} asked if the inequality is
sharp; in more detail, they defined, for $X\subs\Rset$, a dimension
\begin{equation}\label{aDim0}
  \aDim X=\inf\{\updim X\times Y-\updim Y: Y\subs\Rset\},
\end{equation}
noticed that~\eqref{eq:1} yields $\hdim X\leq\aDim X$
and asked if $\aDim X=\hdim X$ for all $X\subs\Rset$.

In 1996 two papers by Bishop and Peres~\cite{MR1376540} and
Xiao~\cite{MR1388205} independently proved that if $X,Y\subs\Rset$ are compact,
then the inequality~\eqref{eq:1} improves to
\begin{equation}\label{eq:2}
  \lpdim X+\updim Y\leq\updim X\times Y,
\end{equation}
doubting thus the conjectured $\aDim X=\hdim X$.
Can one prove~\eqref{eq:2} for arbitrary $X,Y\subs\Rset$ or even in
a more general setting?
The proof in~\cite{MR1376540} is very technical and relies upon
geometry of Euclidean spaces. On the other hand, the Xiao's~\cite{MR1388205} proof
of~\eqref{eq:2} is a rather straightforward and simple application of Baire Category Theorem
and can be thus easily extended to any compact metric spaces; and using the
Joyce and Preiss theorem~\cite{MR1346667}, to analytic metric spaces.
But it seems impossible to exploit the idea any further (actually~\cite{MR2291648}
states~\eqref{eq:2} for arbitrary subsets of the line, but the proof therein is
not very convincing).

It, however, turns out that a much finer, more general and sharper inequality
can be proved in a rather general setting.
Let us outline it in some detail.
The dimensions in~\eqref{eq:2} are,
like many other fractal dimensions, rarefaction indices of fractal measures:
the packing dimension $\updim X$ is the number $s_0$ such that $\pack s(X)=0$
for all $s>s_0$ and $\pack s(X)=\infty$ for all $s<s_0$.
The lower packing dimension is defined likewise from the so called
Hewitt-Stromberg measures $\boxm s(X)$. The inequality~\eqref{eq:2} is
a trivial consequence of the integral inequality
\begin{equation}\label{eq:krapka}
  \pack{s+t}(E)\geq\int\boxm{s}(E_x)\dd\pack{t}(x)
\end{equation}
that holds for any subset $E\subs X\times Y$ of a product of metric spaces.
This inequality, however, does not really help with the solution of the
Hu--Taylor problem.
The crucial step towards its solution is the following
observation. Given a set $E$ in a metric space and $s\geq0$, define
the \emph{lower box content}
$\boxm{s}_0(E)=\liminf_{\del\to0}C_\del(E)/\del^{s}$, where $C_\del(E)$
is the maximal number of points within $E$ that are mutually more than $\del$
apart. The Hewitt-Stromberg measure $\boxm s(E)$ obtains from $\boxm{s}_0$
by the standard \emph{Method I construction}:
$\boxm s(E)=\inf\sum_n\boxm{s}_0(E_n)$, where the infimum is over all countable
covers of $E$. The resulting set function is an outer Borel measure.
Let's define another set function arising from $\boxm{s}_0$ by the formula
$\dboxm s(E)=\inf\sup_n\boxm{s}_0(E_n)$, where the infimum is this time over all
\emph{increasing} covers of $E$. If $\boxm{s}_0$ were, like e.g.~the upper
box content, subadditive, we would get the same values as from Method I.
But it is not. The set function $\dboxm s$ substantially differs from the
Hewitt-Stromberg measure. It is a not measure, it is not even finitely subadditive,
but it turns to be the right mean for solution of the Hu--Taylor problem.
Once one figures out the proof of~\eqref{eq:krapka}, it is easy to improve it to
\begin{equation}\label{eq:krapka2}
  \pack{s+t}(E)\geq\int\dboxm{s}(E_x)\dd\pack{t}(x).
\end{equation}
Consequently the rarefaction index $\dpdim$ of $\dboxm s$ satisfies,
for any metric spaces $X,Y$,
\begin{equation}\label{eq:krapka3}
  \dpdim X+\updim Y\leq\updim X\times Y.
\end{equation}
Since $\dpdim X$ can be easily expressed in terms of lower box dimension
(cf.~Definition~\ref{def:dpdim}), it is not that esoteric.
This improvement of inequality~\eqref{eq:2} gives the best-so-far
lower estimate for the dimension of~\eqref{aDim0}: $\aDim X\geq\dpdim X$ for
all $X\subs\Rset$ and actually for any metric space $X$.

As to the upper estimate of $\aDim$, Xiao~\cite{MR1388205} proved that for any
$X\subs\Rset$, $\aDim X$ is estimated from above by the lower box
dimension. And, luckily, analysis of his proof revealed that one can work upon
its ideas to push the upper estimate down to $\dpdim X$.
Therefore~\eqref{eq:krapka3} is optimal. We arrived at the
solution of Hu--Taylor problem:
\begin{thmun}
\begin{enum}
\item $\dpdim X+\updim Y\leq\updim X\times Y$ holds for any metric spaces $X,Y$.
\item For any $X\subs\Rset^n$ there is a compact set $Y\subs\Rset^n$ such that
$\dpdim X+\updim Y=\updim X\times Y$.
\item In particular, $\inf\{\updim X\times Y-\updim Y: Y\subs\Rset^n\}=\dpdim X$
for all $X\subs\Rset^n$.
\end{enum}
\end{thmun}
Actually, with a proper extension of the definition of $\aDim$, the theorem remains
valid for any space $X$ whose Assouad dimension is finite.

\smallskip
The paper is organized as follows. In Section~\ref{sec:measures} we recall
in detail packing and Hewitt-Stromberg measure.
Then we introduce scaled measures and upper/lower box and packing measures
and list some elementary properties of these measures.
In Section~\ref{sec:int} we state and prove~\eqref{eq:krapka2} and other
integral inequalities involving these measures and derive inequalities for
cartesian rectangles.
In Section~\ref{sec:dim} the notions of upper/lower box and packing dimensions
are recalled and the dimension $\dpdim$ is introduced. Then
we set up and prove~\eqref{eq:krapka3} and other dimension inequalities
following from the respective results of Section~\ref{sec:int}.
Section~\ref{sec:Xiao} is devoted to the solution of the Hu--Taylor problem
in a rather general setting.
The paper is concluded with Section~\ref{sec:rem} that, besides various comments,
presents applications to dimension theory of Borel measures,
and lists some open problems.

\section{The measures}\label{sec:measures}

In this section we set up definitions of packing and box measures whose
behavior on cartesian products is investigated in the next section.
We begin with recalling two common measures --- the packing measure
and the Hewitt-Stromberg measure.
Then we generalize these notions, notice that via this generalization
they are closely related and introduce the lower packing measure and a couple of
more measures and pre-measures.

Since the technique used is rather standard, we present only few brief proofs.

Throughout the section, $X$ stands for a separable metric space with a metric $d$.
Notation used includes $B(x,r)$ for the closed ball of radius $r>0$ centered
at $x\in X$;
$\clos A$ for the closure of a set $A$;
$\abs A$ for the cardinality of a set $A$; and $\nset$
for the set of all natural numbers including zero.

\subsection*{Pre-measures}
It will be  convenient to establish elementary
features of the following constructions of pre-measures from pre-measures.

A \emph{set function} is a mapping $\tau$ that assigns to each $E\subs X$ a value
$\tau(E)\in[0,\infty]$. The notions of \emph{monotone}, \emph{subadditive}
and \emph{countably subadditive} set function are self-explaining.
A set function will be called a \emph{pre-measure} if it is monotone
and $\tau(\emptyset)=0$.
A pre-measure $\tau$ is \emph{metric} if $\tau(A\cup B)\geq\tau(A)+\tau(B)$
whenever $A,B\subs X$ are \emph{separated}, i.e.~$\dist(A,B)>0$.
Departing slightly from the common usage we call a pre-measure an
\emph{outer measure} if its restriction
to the algebra of Borel sets is a Borel measure.
An outer measure $\tau$ is \emph{Borel-regular} if for each $A\subs X$ there is
$B\supseteq A$ Borel with $\tau(B)=\tau(A)$.

The following is the Munroe's \emph{Method I construction},
see~\cite{MR0281862}. Its point is that it produces a countably subadditive
pre-measure from any pre-measure.
Given a pre-measure $\tau$, the new pre-measure $\widehat\tau$ is defined by
$$
  \widehat\tau(E)=\inf\left\{\sum\nolimits_n
  \tau(E_n):E\subs\bigcup\nolimits_n E_n\right\}.
$$
We shall also make use of a ``directed'' variation of
\emph{Method I}. Write $E_n\upto E$ to denote that $\seq{E_n}$ is an
increasing sequence of sets with union $E$.
$$
  \arrow{\tau}(E)=\liminf_{E_n\upto E}\tau(E_n)=
  \inf\left\{\sup\nolimits_n
  \tau(E_n):E_n\upto E\right\}.
$$
Let us call this construction \emph{Method D} for future reference.
We list some elementary properties of the two operations.

\begin{lem}\label{lem:MI}
Let $\tau$ be a metric pre-measure on $X$.
\begin{enum}
\item
$\widehat{\tau}$ is an outer measure.
\item
If $\tau$ is Borel-regular, then so is $\widehat\tau$ and
$\widehat{\tau}\leq\arrow{\tau}$.
\item
If $\tau$ is subadditive, then $\arrow{\tau}=\widehat{\tau}$.
\end{enum}
\end{lem}
\begin{proof}
(i) By~\cite[Theorem 4]{MR0281862}, $\widehat\tau$ is countably subadditive.
It is easy to check that since $\tau$ is metric, so is $\widehat\tau$.
Hence (i) follows by~\cite[Theorem 19]{MR0281862}.

(ii)
It is obvious that since $\tau$ is Borel-regular, so is $\widehat\tau$.
It is also obvious that $\widehat\tau\leq\tau$; thus
$\arrow{\widehat\tau}\leq\arrow\tau$.
As $\widehat\tau$ is a Borel-regular outer measure,
$\sup_n\widehat\tau(E_n)=\widehat\tau(E)$ holds for any sequence
$E_n\upto E$, cf.~\cite[Theorem 4]{MR0281862}. Hence
$\widehat\tau\leq\arrow{\widehat\tau}$ and
$\widehat\tau\leq\arrow\tau$ follows.

(iii) Let $E_n\upto E$. Set $A_0=E_0$ and $A_n=E_n\setminus E_{n-1}$
for $n>0$. Then $A_n$'s cover $E$ and by assumption,
$\sup_n\tau(E_n)\leq\sup_n\sum_{i\leq n}\tau(A_i)=\sum_{n\in\nset}\tau(A_n)$.
This yields $\arrow\tau\leq\widehat\tau$. The opposite
inequality follows from (ii).
\end{proof}

We now recall two classical measures: the packing measure and the
Hewitt-Stromberg measure. They will play an important role in our considerations and
moreover will motivate our definitions of scaled measures.

\subsection*{Packing measure}
There are perhaps too many notions of packing and packing measure.
We choose the one used e.g.~in~\cite{MR1362951} and~\cite{MR1346667}; the other definitions
are briefly discussed in Section~\ref{sec:rem}.
A family $\{(x_i,r_i):i\in I\}\subs X\times(0,\infty)$ is called a \emph{packing} if
$x_i\notin B(x_j,r_j)$ for all $i\neq j$ in $I$.
Equivalently, if $d(x_i,x_j)>r_i$.
It is a \emph{packing  of a set} $E\subs X$ if
$x_i\in E$ for all $i\in I$.
It is \emph{$\del$-fine} if $r_i\leq\del$ for all $i\in I$.

We shall need the following simple lemma at a couple of instances.

\begin{lem}\label{shift}
For any finite packing $\{(x_i,r_i):i\in I\}$ there is $\eps>0$ such
that $\{(x'_i,r'_i):i\in I\}$ is a packing whenever $d(x'_i,x_i)<\eps$
and $r'_i<r_i+\eps$ for all $i\in I$.
\end{lem}
\begin{proof}
It is enough to put $\eps=\frac13\min\{d(x_i,x_j)-\max(r_i,r_j):i\neq j\}$.
\end{proof}
Following~\cite{MR1362951}, a \emph{Hausdorff function} is a nondecreasing function
$g:(0,\infty)\to(0,\infty)$.
No continuity of $g$ is \emph{a priori} imposed.
Hausdorff functions are (partially) ordered by $f\prec g$ iff
$\varlimsup_{r\to0}f(r)/g(r)=0$.

If $g$ is a Hausdorff function and $\pi=\{(x_i,r_i):i\in I\}$ a packing,
we write
$$
  g(\pi)=\sum\{g(r_i):i\in I\}.
$$

\begin{defn}[{\cite{MR1362951,MR1346667}}]\label{packclassic}
Let $g$ be a Hausdorff function and $E\subs X$. Let
$$
  \pack{g}_0(E)=\inf_{\del>0}\pack{g}_\del(E),\quad
  \text{where }
  \pack{g}_\del=\sup\{g(\pi):
  \text{$\pi$ is a $\del$-fine packing of $E$}\}.
$$
The \emph{$g$-dimensional packing measure} of $E$
is defined by $\pack{g}(E)=\widehat{\pack{g}_0}(E)$.

In the particular case when $g(r)=r^s$ for some constant $s\geq0$, we write,
as usual, $\pack{s}$ instead of $\pack{g}$; and the same license is used for
other pre-measures and measures obtained from Hausdorff functions.
\end{defn}

It is well-known and easy to see that $\pack{g}_0$ is an additive Borel-regular
metric pre-measure and thus $\pack{g}$ is a Borel-regular outer measure.

\subsection*{Hewitt-Stromberg measure}
For $F\subs X$ define $\gap F=\inf\{d(x,y):x,y\in F,\ x\neq y\}$.
For $E\subs X$ and $\del>0$ denote
\begin{equation}\label{eq:capacity0}
  C_\del(E)=\sup\{\card{F}:F\subs E, \gap F>\del\}
\end{equation}
the \emph{$\del$-capacity} of $E$.
The following natural notion appeared first in~\cite[(10.51)]{MR0188387}.
It was investigated and got the name in~\cite{MR827889,MR918685}.
Another excellent reference is~\cite{MR1484412}.
\begin{defn}[{\cite{MR0188387}}]
Let $g$ be a Hausdorff function and $E\subs X$. Let
\begin{equation}\label{capac2}
  \boxm{g}_0(E)=\liminf_{\del\to0}C_\del(E)g(\del).
\end{equation}
The \emph{$g$-dimensional Hewitt-Stromberg measure} of $E$ is defined by
$\boxm{g}(E)=\widehat{\boxm{g}_0}(E)$.
\end{defn}
It is easy to check that $\boxm{g}_0$ is a Borel-regular metric pre-measure
(though is does not have to be subadditive)
and thus $\boxm{g}$ is a Borel-regular outer measure.
%
%

\subsection*{Scaled measures}
A set $\Del\subs(0,\infty)$ such that $0\in\clos\Del$ is termed a \emph{scale}.
We use $\Del$ as a generic symbol for a scale.
A packing $\{B(x_i,r_i):i\in I\}$ is \emph{$\Del$-valued} if $r_i\in\Del$ for all
$i\in I$.
A \emph{$(\Del,\del)$-packing} is a packing that is
$\Del\cap(0,\del]$-valued, i.e.~$\Del$-valued and $\del$-fine.
A packing $\{B(x_i,r_i):i\in I\}$ is \emph{uniform} if $r_i=r_j$ for
all $i,j\in I$.

We now introduce an auxiliary notion of a $\Del$-scaled packing measure.
It is a straight generalization of the packing measure,
the only difference is that the radii allowed in packings are
limited to the set $\Del$.
\begin{defn}
Let $g$ be a Hausdorff function, $\Del$ a scale and $E\subs X$.
Let $\pack{g}_{\Del,0}(E)=\inf_{\del>0}\pack{g}_{\Del,\del}(E)$, where
$$
  \pack{g}_{\Del,\del}(E)=\sup\{g(\pi):
  \text{$\pi$ is a $(\Del,\del)$-packing of $E$}\}.
$$
The \emph{$g$-dimensional $\Del$-packing measure} of $E$
is defined by $\pack{g}_{\Del}(E)=\widehat{\pack{g}_{\Del,0}}(E)$.

Let $\boxm{g}_{\Del,0}(E)=\inf_{\del>0}\boxm{g}_{\Del,\del}(E)$, where
$$
  \boxm{g}_{\Del,\del}(E)=\sup\{g(\pi):
  \text{$\pi$ is a uniform $(\Del,\del)$-packing of $E$}\}.
$$
The \emph{$g$-dimensional $\Del$-box measure} of $E$ is defined by
$\boxm{g}_{\Del}(E)=\widehat{\boxm{g}_{\Del,0}}(E)$.
\end{defn}
Clearly $\boxm{g}_{\Del,0}\leq\pack{g}_{\Del,0}$ and $\boxm{g}_{\Del}\leq\pack{g}_{\Del}$.
It is easy to check that the set function $\boxm{g}_{\Del,0}$
can be equivalently defined in terms of capacity:
\begin{equation}\label{capdel}
  \boxm{g}_{\Del,0}(E)=\limsup_{\del\in\Del,\del\to0}C_\del(E)\cdot g(\del).
\end{equation}
This equation shows the link to the Hewitt-Stromberg measure.

Here are some elementary facts about the scaled measures.
(i) and (iv) are obvious, (ii) is a consequence
of Lemma~\ref{shift} and (iii) follows from Lemma~\ref{lem:MI}.
\begin{lem}\label{lem:Del}
\begin{enum}
\item $\pack{g}_{\Del,0}$ and $\boxm{g}_{\Del,0}$ are subadditive metric pre-measures,
\item $\pack{g}_{\Del,0}(E)=\pack{g}_{\Del,0}(\clos E)$ for any set $E\subs X$,
  and likewise for $\boxm{g}_{\Del,0}$.
\item $\pack{g}_\Del$ and $\boxm{g}_\Del$ are Borel regular outer measures.
\end{enum}
\end{lem}

\subsection*{Upper measures}
We now define upper packing and box measures as extreme cases of
corresponding scaled measures. Among all scales, $(0,\infty)$ is the largest one.
The corresponding scaled measures are thus largest among all scaled measures.
\begin{defn}
Let $g$ be a Hausdorff function and $E\subs X$. Let
\begin{align*}
  \upack{g}_0(E)&=\sup_{\Del}\pack{g}_{\Del,0}(E)=\pack{g}_{(0,\infty)}(E),\\
  \uboxm{g}_0(E)&=\sup_{\Del}\boxm{g}_{\Del,0}(E)=\boxm{g}_{(0,\infty)}(E).
\end{align*}
The \emph{$g$-dimensional upper packing and box measures}
of $E$ are defined, respectively, by
$\upack{g}(E)=\widehat{\upack{g}_0}(E)$ and
$\uboxm{g}(E)=\widehat{\uboxm{g}_0}(E)$.
\end{defn}
It is clear that the upper packing measure $\upack{g}$ is nothing but the
packing measure $\pack{g}$ as defined in~\ref{packclassic}.
We defined it just to point out the duality of (upper) packing measure
and lower packing measure defined below. We prefer notation $\upack{g}$ to make clear
distinction between the upper and lower packing measures.
As to $\uboxm{g}_0$, it follows from \eqref{capdel} that
\begin{equation}\label{capac3}
  \uboxm{g}_0(E)=\limsup_{\del\to0}C_\del(E)g(\del)
\end{equation}
and thus $\uboxm{g}$ is via~\eqref{capac2} dual to the Hewitt-Stromberg measure.
The upper box measures $\uboxm s$ appear in many papers and books,
explicitly e.g.~in~\cite{MR1489237} and implicitly e.g.~in~\cite[5.3]{MR1333890}.

Note that $\upack{g}$ and $\uboxm{g}$ and the underlying pre-measures satisfy
Lemma~\ref{lem:Del}.

\subsection*{Lower measures}
Likewise we define lower packing and box measures as the lower extreme cases of
corresponding scaled measures. The situation is more delicate, since there is no
minimal scale.

\begin{defn}
Let $g$ be a Hausdorff function and $E\subs X$. Let
$$
  \lpack{g}_0(E)=\inf_{\Del}\pack{g}_{\Del,0}(E),\qquad
  \lboxm{g}_0(E)=\inf_{\Del}\boxm{g}_{\Del,0}(E),
$$
the infima over all scales.
The \emph{$g$-dimensional lower packing and box measures}
of $E$ are defined, respectively, by
$\lboxm{g}(E)=\widehat{\lboxm{g}_0}(E)$ and
$\uboxm{g}(E)=\widehat{\uboxm{g}_0}(E)$.

Since the upper pre-measures $\upack{g}_0$ and $\uboxm{g}_0$ are subadditive,
\emph{Method D} yields the same measures as \emph{Method I}.
It, however, is not the case of lower measures. That is why we
also define $\dpack{g}(E)=\arrow{\lpack{g}_0}(E)$ and
$\dboxm{g}(E)=\arrow{\lboxm{g}_0}(E)$.
\end{defn}
It follows from \eqref{capdel} that
$\lboxm{g}_0(E)=\liminf_{\del\to0}C_\del(E)g(\del)$.
Thus $\lboxm{g}_0=\boxm{g}_0$ and $\lboxm{g}=\boxm{g}$, i.e.~the
lower box measure is just another name for the Hewitt-Stromberg measure.
The lower packing measure and the two directed pre-measures seem to be new
concepts.

\begin{lem}\label{basic0}
\begin{enum}
\item $\lpack{g}_0$ and $\lboxm{g}_0$ are metric pre-measures,
\item $\lpack{g}_0(E)=\lpack{g}_0(\clos E)$ for any set $E\subs X$,
  and likewise for $\lboxm{g}_0$,
\item $\lpack{g}$ and $\lboxm{g}$ are Borel-regular outer measures,
\item if $\lpack{g}_0(E)<\infty$, then $E$ is totally bounded,
  and likewise for $\lboxm{g}_0$.
\end{enum}
\end{lem}
\begin{proof}
(i) is straightforward, (ii) follows from Lemma~\ref{lem:Del}(ii) and (iii)
is a consequence of (ii). To prove (iv) it is enough to notice that if
$\lboxm{g}_0(E)<\infty$, then by~\eqref{capac2} $C_\del(E)<\infty$ for all $\del>0$.
\end{proof}

\begin{lem}\label{basic}
For any Hausdorff function $g$
\begin{enum}
\item $\lpack{g}_0(E)=
  \inf_{\Del}\sup\{g(\pi):\text{$\pi$ is a $\Del$-valued packing of $E$}\}$,
\item $\lboxm{g}_0(E)=
  \inf_{\Del}\sup\{g(\pi):\text{$\pi$ is a $\Del$-valued uniform packing of $E$}\}$,
\item $\dpack{g}\leq \inf_{\Del}\pack{g}_{\Del}$,
\item $\dboxm{g}= \inf_{\Del}\boxm{g}_{\Del}$.
\end{enum}
\end{lem}
\begin{proof}
(i) For each scale $\Del$ denote
$S_\Del=\sup\{g(\pi):\text{$\pi$ is a $\Del$-valued packing of $E$}\}$ and
$S=\inf_\Del S_\Del$. Note that if $\Del$ is a scale and $\del>0$, then
$\Del\cap(0,\del)$ is also a scale.
By the definition, $\pack{g}_{\Del,0}(E)=\inf_{\del>0}S_{\Del\cap(0,\del)}\geq S$,
which in turn yields $\lpack{g}_0(E)\geq S$. The reverse inequality is obvious.
(ii) is proved in the same manner.

(iii) Clearly $\lpack g_0\leq\pack g_{\Del,0}$. Hence
$\dpack g\leq\arrow{\pack g_{\Del,0}}$ and since Lemma~\ref{lem:MI}(iii)
yields $\arrow{\pack g_{\Del,0}}=\pack g_\Del$, we are done.

(iv) $\dboxm{g}\leq\inf_{\Del}\boxm{g}_{\Del}$ is proved the same way as (iii).
To prove the reverse inequality, let $E\subs X$ and $s>\dboxm{g}(E)$. There is
$E_n\upto E$ such that $\lboxm g_0(E_n)<s$ for all $n$, i.e.~there are scales $\Del_n$
such that $C_{E_n}(r)g(r)<r$ for all $n$ and $r\in\Del_n$.
Choose $r_n\in\Del_n$ so that the resulting sequence decreases to zero
and let $\Del=\{r_n:n\in\nset\}$.
Proving that $\boxm g_{\Del,0}(E_n)\leq s$ for all $n$,
and thus $\boxm g_\Del(E)\leq s$, is straightforward.
\end{proof}
We do not know if the inequality (iii) can be reversed.

\subsection*{Comparison}
The inequalities $\lpack{g}\leq\dpack{g}\leq\lpack{g}_0$ and
$\lboxm{g}\leq\dboxm{g}\leq\lboxm{g}_0$ are trivial. As follows from Example~\ref{ex:3x},
none of these four inequalities can be reversed.
It is also clear that $\lboxm{g}_0\leq\lpack{g}_0$, $\lboxm{g}\leq\lpack{g}$ and
$\dboxm{g}\leq\dpack{g}$, but we do not know if they can be reversed.
We only know that $\lboxm{g}_0$ and $\lpack{g}_0$ have the same null sets.
\begin{prop}\label{lem:00}
For any set $E\subs X$,
$\lpack{g}_0(E)=0$ if and only if $\lboxm{g}_0(E)=0$.
\end{prop}
\begin{proof}
The forward implication is obvious.
To prove the backward one assume $\lboxm{g}_0(E)=0$.
Then there is a sequence $r_n\downarrow0$ such that
$C_{r_n}(E)g(r_n)\leq 2^{-n}$.
For $m\in\nset$ define a scale $\Del_m=\{r_n:n>m\}$.
If $\pi$ is a $\Del_m$-valued packing, then
$$
  g(\pi)=\sum_{n>m}\sum\{g(r_n):(x,r_n)\in\pi\}
  \leq\sum_{n>m}C_{r_n}(E)g(r_n)
  \leq\sum_{n>m}2^{-n}=2^{-m}.
$$
Therefore $\lpack{g}_{\Del_m,0}(E)\leq2^{-m}$ and consequently
$\lpack{g}_0(E)\leq\inf_{m\in\nset}2^{-m}=0$.
\end{proof}
This proposition is enough to show that the measures $\lboxm{g}$ and $\lpack{g}$
are close to each other:
\begin{prop}\label{comp2}
The following are equivalent:
\begin{enum}
\item there is a countable cover $\{E_n\}$ of $E$ such that
$\lboxm{g}_0(E_n)=0$ for all $n$,
\item there is $h\prec g$ such that $\lboxm{h}(E)=0$,
\item there is $h\prec g$ such that $\lpack{h}(E)=0$.
\end{enum}
\end{prop}
\begin{proof}
(i)$\Leftrightarrow$(ii) is proved in~\cite[Prop.~6]{MR827889}.
The proof therein can be easily adapted to show that, via Proposition~\ref{lem:00},
(i)$\Rightarrow$(iii), and (iii)$\Rightarrow$(ii) is obvious.
\end{proof}
The directed pre-measures are also close:
\begin{prop}\label{comp3}
The following are equivalent:
\begin{enum}
\item there is $E_n\upto E$ such that
$\lboxm{g}_0(E_n)=0$ for all $n$,
\item there is $E_n\upto E$ and a sequence $r_n\downarrow0$ such that
$C_{r_n}(E_n)g(r_n)\to0$,
\item there is $h\prec g$ such that $\dboxm{h}(E)=0$,
\item there is $h\prec g$ such that $\dpack{h}(E)=0$,
\item there is $h\prec g$ and a scale $\Del$ such that $\pack{h}_\Del(E)=0$.
\end{enum}
\end{prop}
\begin{proof}
(i)$\Rightarrow$(ii): By virtue of~\eqref{capdel} there is, for each $n$,
a sequence $r_i^n\downarrow0$ such that $r_n^n\to0$ and
$\lim_{i\to\infty}C_{r_i^n}(E_n)g(r_i^n)=0$.
It is enough set $r_n=r_n^n$.

(ii)$\Rightarrow$(v):
Since $\lim_{n\to\infty}C_{r_n}(E_n)g(r_n)=0$, there is clearly $h\prec g$ such that
$\lim_{n\to\infty}C_{r_n}(E_n)h(r_n)=0$.
Letting $\Del=\{r_n:n\in\nset\}$ we have $\pack{h}_\Del(E)=0$.

(v)$\Rightarrow$(iv)$\Rightarrow$(iii) are obvious and (iii)$\Rightarrow$(i)
follows from~\ref{comp4}(i) below.
\end{proof}
\begin{prop}\label{comp4}
If $g\prec h$, then for any set $E\subs X$
\begin{enum}
\item $\lboxm{g}_0(E)<\infty\,\,\Rightarrow\lpack{h}_0(E)=0$,
\item $\lboxm{g}(E)<\infty\,\,\Rightarrow\lpack{h}(E)=0$,
\item $\dboxm{g}(E)<\infty\Rightarrow\dpack{h}(E)=0$
and there is a scale $\Del$ such that $\pack{h}_\Del(E)=0$.
\end{enum}
\end{prop}
\begin{proof}
(i) Using~\eqref{capdel} it is clear that if $g\prec h$, then $\lboxm{g}_0(E)<\infty$
yields $\lboxm{h}_0(E)=0$. Now use Proposition~\ref{lem:00}.
(ii) is an obvious consequence of (i). (iii) Suppose $\dboxm{g}(E)<\infty$.
Then, by (i), condition (i) of the above proposition is satisfied. Hence also
conditions~\ref{comp3}(iv) and (v) are satisfied, which is enough.
\end{proof}
These three propositions show that $\lboxm{g}$ and $\lpack{g}$, as well as
$\dboxm{g}$ and $\dpack{g}$, respectively, are in a sense very close,
as contrasted by the corresponding upper measures:
As to the comparison of $\uboxm{g}$ and $\upack{g}$,
needless to say that $\uboxm{g}\leq\upack{g}$, but not much more can be said,
except that if $0<s<t$ and $\uboxm{s}(E)=0$, then $\upack{t}(E)=0$.
This fact can be extracted e.g.~from the proof of~\cite[Theorem 5.11]{MR1333890}.
It, however, is not difficult to show by
example that the upper measures fail statements analogical to
Propositions~\ref{lem:00}, \ref{comp2} and~\ref{comp4}.

\subsection*{Lipschitz maps}
All of the (pre)-measures under consideration behave under Lipschitz maps as
expected. The simple proof of the following is omitted.
\begin{lem}\label{lipp}
Let $s\geq 0$ and let $f:X\to Y$ be a $c$-Lipschitz map.
Then $\lpack s(f(X))\leq c^s\lpack s(X)$ and likewise for
$\lboxm s_0$, $\uboxm s_0$, $\lpack s_0$, $\upack s_0$,
$\lboxm s$, $\dboxm s$, $\uboxm s$, $\dpack s$ and $\upack s$.
\end{lem}

\section{Packing measures on cartesian products}\label{sec:int}

This section is devoted to investigation of integral and product inequalities
involving packing and box measures.
Fix two metric spaces $X,Y$ and provide their cartesian product
$X\times Y$ with the maximum metric.
For a set $E\subs X\times Y$ and $x\in X$, the cross section
$\{y\in Y:(x,y)\in E\}$ is denoted $E_x$ or $(E)_x$.
Fix also a scale $\Del$ and two Hausdorff functions $g,h$.

\begin{lem}\label{lem:0-0}
For any set $E\subs X\times Y$
$$
  \pack{gh}_{\Del,0}(E)
  \geq\pack{h}_\Del(X)\cdot\inf_{x\in X}\lboxm{g}_0(E_x).
$$
\end{lem}
\begin{proof}
Let $c<\inf_{x\in X}\lboxm{g}_0(E_x)$.
For each $x$ there is a number $n\in\nset$ such that
$\lboxm{g}_\del(E_x)>c$ for all $\del<\frac{1}{n}$. Setting
$$
  B_n=\{x:\lboxm{g}_\del(E_x)>c\text{ for all $\del<\tfrac1n$}\}
$$
we thus have $B_n\upto X$.
Let $d<\pack{h}_{\Del}(X)$.
Lemmas~\ref{lem:Del}(iii) and~\ref{lem:MI}(iii) yield
$n$ such that $d<\pack{h}_{\Del,0}(B_n)$.
Hence there is $\del_0>0$ such that for all $\del<\del_0$ there is a
$(\Del,\del)$-packing
$\pi=\{(x_i,r_i):i\in I\}$ of $B_n$ such that $h(\pi)>d$.
We may assume $\del_0<\frac1n$. Thus for each $i\in I$ there is a
uniform $\del$-fine packing $\pi_i=\{(y_{ij},r_i):i\in K_i\}$ of
$E_{x_i}$ such that
$
  g(\pi_i)=\card{K_i}\cdot g(r_i)>c.
$
The collection
$
  \sigma=\bigl\{\bigl((x_i,y_{ij}),r_i\bigr):i\in I,j\in K_i\bigr\}
$
is thus a $(\Del,\del)$-packing of $A$ and
$$
  gh(\sigma)=\sum_{i\in I}\sum_{j\in K_i}g(r_i)h(r_i)=
  \sum_{i\in I}\card{K_i}g(r_i)h(r_i)>
  c\sum_{i\in I}h(r_i) >cd.
$$
Therefore $\pack{gh}_{\Del,\del}(A)>cd$.
As this holds for any $\del<\del_0$ and all $d<\pack{h}_{\Del}(X)$ and
$c<\inf_{x\in X}\lboxm{g}_0(E_x)$,
we are done.
\end{proof}

\begin{lem}\label{lem:cpt}
If $E\subs X\times Y$ is compact,
then the mapping $x\mapsto\lboxm{g}_0(E_x)$
is Borel measurable and
\begin{equation}\label{eq:lemma}
  \pack{gh}_{\Del,0}(E)
    \geq\int\lboxm{g}_0(E_x)\dd\pack{h}_{\Del}(x).
\end{equation}
\end{lem}
\begin{proof}
We first show that if $E$ is compact, then $x\mapsto\lboxm{g}_0(E_x)$
is Borel measurable.
It follows from Lemma~\ref{shift} that
\begin{equation}\label{capdelx}
    \forall r>0\ \exists \eps>0\ C_{r+\eps}(E)=C_r(E).
\end{equation}
Thus for $x\in X$ fixed, the mapping
$\del\mapsto C_\del(E_x)$ is right-continuous. Therefore
the mapping $\del\mapsto C_\del(E_x)\cdot g(\del)$
is right-continuous at each point of (right-)continuity of $g$.
So if $Q\subs(0,\infty)$ is a dense countable set and
$D$ the set of points of discontinuity of $g$, then
$$
  \lboxm{g}_0(E_x)=
  \liminf_{\substack{\del\to0\\ \del\in Q\cup D}}C_\del(E_x)\cdot g(\del).
$$
As $g$ is nondecreasing, the set $D$ is countable.
Therefore $x\mapsto\lboxm{g}_0(E_x)$ obtains from a countable family
of mappings of the form
$$
  x\mapsto C_\del(E_x)\cdot g(\del),\quad \del\in Q\cup D.
$$
Borel measurability of $x\mapsto\lboxm{g}_0(E_x)$ will thus follow
if we show that each of these mappings is Borel measurable. To that
end we prove that
for any $\del>0$ and each integer $n$ the set
$
  L=\{x\in X:C_\del(E_x)\geq n\}
$
is Borel, which is enough, as $g(\del)$ is constant and $C_\del(E_x)$ is
integer-valued.
For each $\eps>0$ set
$$
  L(\eps)=
  \{x:\text{there is $\{y_1,y_2,\dots,y_n\}\subs E_x$,
 $\gap\{y_1,y_2,\dots,y_n\}\geq\del+\eps$}\}.
$$
By~\eqref{capdelx}, $L=\bigcup_{\eps>0}L(\eps)$.
Each of the sets $L(\eps)$ is closed:
Let $x_k\to x$ be a sequence in $L(\eps)$ and
$\{y_{k1},y_{k2},\dots,y_{kn}\}\subs E_{x_k}$ sets witnessing $x_k\in L(\eps)$.
Choosing a subsequence if necessary, for each $i\leq n$ the sequence
$(x_k,y_{ik})$ converges to a point $(x,y_i)\in E$;
this follows from compactness of $E$. The set $\{y_1,y_2,\dots,y_n\}$
obviously witnesses $x\in L(\eps)$.
So $L(\eps)$ is closed and therefore
$L=\bigcup_{m\in\nset}L(1/m)$ is $F_\sigma$ and hence Borel.

The next goal is to derive~\eqref{eq:lemma} from Lemma~\ref{lem:0-0}.
As $E$ is compact, replacing $X$ and $Y$ with projections of $E$ we
may assume both $X,Y$ compact.
Write $\mu=\pack{h}_{\Del}$.

We need to show that $\pack{gh}_{\Del,0}(E)\geq\int s\dd\mu$
for each simple function $s\leq\lboxm{g}_0(E_x)$.
Let $s=\sum_{i=1}^m c_i\chi_{A_i}$ be such a function, with $A_i$'s
disjoint Borel sets and $c_i$'s positive.
If there is $i$ such that $\mu(A_i)=\infty$, then
$
  \pack{gh}_{\Del,0}(E)\geq
  c_i\mu(A_i)=\infty
$
by Lemma~\ref{lem:0-0}.
Otherwise $\mu(A_i)<\infty$ for all $i$ and thus $A_i$'s may be
approximated from within with compact sets, for $X$ is compact:
For any $\eps>0$ and each $i$ there is a compact set $K_i\subs A_i$
such that $\mu(K_i)>\mu(A_i)-\frac{\eps}{mc_i}$. Therefore
$$
  \int s\dd\mu=\sum_{i=1}^m c_i\mu(A_i)\leq
  \sum_{i=1}^m c_i\left(\mu(K_i)+\frac{\eps}{mc_i}\right)=
  \eps+\sum_{i=1}^m c_i\mu(K_i).
$$
For each $i$ put $E_i=E\cap(K_i\times Y)$.
Apply Lemma~\ref{lem:0-0} to $E_i$'s to get
$$
  c_i\mu(K_i)\leq \pack{gh}_{\Del,0}(E_i),\quad i=1,2,\dots,m.
$$
Thus
$
  \int s\dd\mu\leq\eps+\sum_{i=1}^m\pack{gh}_{\Del,0}(E_i).
$
As $K_i$'s are disjoint compacta, so are $E_i$'s. Therefore $E_i$'s, being disjoint,
are separated and thus Lemma~\ref{lem:Del}(i) yields
$$
  \sum_{i=1}^m\pack{gh}_{\Del,0}(E_i)=
  \pack{gh}_{\Del,0}\left(\bigcup\nolimits_{i=1}^mE_i\right)
  \leq\pack{gh}_{\Del,0}(E).
$$
Therefore $\int s\dd\mu\leq\eps+\pack{gh}_{\Del,0}(E)$.
Since $\eps>0$ and $s\leq\lboxm{g}_0(E_x)$ were arbitrary,
\eqref{eq:lemma} follows.
\end{proof}

Since the mappings $x\mapsto\lboxm{g}_0(E_x)$,
$x\mapsto\lboxm{g}(E_x)$, $x\mapsto\dboxm{g}_0(E_x)$ etc.~%
need not be Borel measurable, we set up the following theorems in terms of
the \emph{upper integral}
$$
 \upint f\dd\mu
 =\inf\left\{\int\phi\dd\mu:
 \phi\geq f\text{ Borel measurable}\right\}.
$$
\begin{lem}\label{lem:delta0}
For any set $E\subs X\times Y$
$$
  \pack{gh}_{\Del,0}(E)
    \geq\upint\lboxm{g}_0(E_x)\dd\pack{h}_{\Del}(x).
$$
\end{lem}
\begin{proof}
As all quantities are intrinsic properties of $E$, \emph{mutatis mutandis}
we may assume $X,Y$ be complete metric spaces.
If $\pack{gh}_{\Del,0}(E)=\infty$, there is nothing to prove.
If $\pack{gh}_{\Del,0}(E)<\infty$, then $E$ is by Lemma~\ref{basic0}(iv)
totally bounded.
Therefore its closure $\clos E$ is compact: for $X\times Y$ is complete.
Hence Lemma~\ref{lem:cpt} yields,
with the aid of Lemma~\ref{lem:Del}(ii),
$$
 \pack{gh}_{\Del,0}(E)=\pack{gh}_{\Del,0}(\clos E)
 \geq\int\lboxm{g}_0\bigl((\clos E)_x\bigr)\dd\pack{h}_{\Del}(x)
 \geq\upint\lboxm{g}_0(E_x)\dd\pack{h}_{\Del}(x).
 \qedhere
$$
\end{proof}
\begin{thm}\label{thm:intpack}
Let $X,Y$ be metric spaces. For any set $E\subs X\times Y$
$$
  \pack{gh}_{\Del}(E)
    \geq\upint\dboxm{g}(E_x)\dd\pack{h}_{\Del}(x).
$$
\end{thm}
\begin{proof}
Let $E_n\upto E$. By Lemma~\ref{lem:delta0},
$\pack{gh}_{\Del,0}(E_n)
    \geq\upint\,\lboxm{g}_0(E_n)_x\dd\pack{h}_{\Del}(x)$
for each $n$. Therefore Levi's monotone convergence theorem yields
\begin{equation}\label{eq:levi}
  \sup_n\pack{gh}_{\Del,0}(E_n)\geq
    \upint\sup_n\lboxm{g}_0(E_n)_x\dd\pack{h}_{\Del}(x)\geq
    \upint\dboxm{g}(E)_x\dd\pack{h}_{\Del}(x),
\end{equation}
because $(E_n)_x\upto E_x$ for all $x\in X$.
Take the infimum over all sequences $E_n\upto E$
to get $\arrow{\pack{g}_{\Del,0}}(E)\geq
\upint\dboxm{g}(E)_x\dd\pack{h}_{\Del}(x)$.
By Lemmas~\ref{lem:Del}(iii) and~\ref{lem:MI}(iii),
$\arrow{\pack{g}_{\Del,0}}(E)=\pack{g}_\Del(E)$.
\end{proof}
The main theorem of this section follows.
\begin{thm}\label{thm:lpack}
Let $X,Y$ be metric spaces. For any set $E\subs X\times Y$
\begin{enum}
\item  $\displaystyle\upack{gh}(E)
    \geq\upint\dboxm{g}(E_x)\dd\upack{h}(x)$,
\item  $\displaystyle\lpack{gh}(E)
    \geq\upint\lboxm{g}(E_x)\dd\lpack{h}(x)$,
\item  $\displaystyle\dpack{gh}(E)
    \geq\upint\dboxm{g}(E_x)\dd\lpack{h}(x)$,
\item  $\displaystyle\dpack{gh}(E)
    \geq\inf_{x\in X}\dboxm{g}(E_x)\cdot\dpack{h}(X)$.
\end{enum}
\end{thm}
\begin{proof}
(i) is a particular case of the above theorem with $\Del=(0,\infty)$.

(ii):
Let $E\subs\bigcup_nE_n$.
Use Lemma~\ref{lem:delta0} for each $n$ and take infima over all scales,
first on the right and then on the left, to get
$\lpack{gh}_0(E_n)\geq\upint\lboxm{g}_0(E_n)_x\dd\lpack{h}(x)$.
Thus by Lebesgue Theorem
$$
  \sum_n\lpack{gh}_0(E_n)\geq
    \int^*\sum_n\lboxm{g}_0(E_n)_x\dd\lpack{h}(x)\geq
    \upint\lboxm{g}(E)_x\dd\lpack{h}(x).
$$
Take the infimum over all sequences $E_n$ such that $E\subs\bigcup_nE_n$
to get the required inequality.

(iii):
Let $E_n\upto E$. As above,
$\lpack{gh}_0(E_n)\geq\upint\lboxm{g}_0(E_n)_x\dd\lpack{h}(x)$.
Now proceed as in~the proof of Theorem~\ref{thm:intpack}, using
Levi's monotone convergence theorem.
(iv) can be proved in the same manner.
\end{proof}

Letting $E=X\times Y$, we get the following estimates for
cartesian rectangles.
The last inequality follows by analysis of the proof of Lemma~\ref{lem:0-0}.
\begin{coro}\label{int:rec2}
For any metric spaces $X,Y$
\begin{enum}
\item  $\upack{gh}(X\times Y)\geq \upack{h}(X)\cdot\dboxm{g}(Y)$,
\item  $\dpack{gh}(X\times Y)\geq \dpack{h}(X)\cdot\dboxm{g}(Y)$,
\item  $\lpack{gh}(X\times Y)\geq \lpack{h}(X)\cdot\lboxm{g}(Y)$,
\item  $\lpack{gh}_0(X\times Y)\geq \lpack{h}_0(X)\cdot\lboxm{g}_0(Y)$.
\end{enum}
\end{coro}
A number of consequences can be derived from these theorems. As a sample we prove
an estimate of packing measure of a domain of a Lipschitz mapping, similar to
\cite[7.7]{MR1333890}.
\begin{coro}\label{coroLip}
Let $X,Y$ be metric spaces and $c\geq1$. Let $f:X\to Y$ be a $c$-Lipschitz map.
For any $s,t\geq0$
\begin{align*}
  \upint\dboxm{t}\left(f^{-1}(y)\right)\dd\upack{s}(y)&\leq c^{s+t}\upack{s+t}(X),\\
  \upint\lboxm{t}\left(f^{-1}(y)\right)\dd\lpack{s}(y)&\leq c^{s+t}\lpack{s+t}(X).
\end{align*}
\end{coro}
\begin{proof}
Let $E=\{(x,f(x)):x\in X\}\subs X\times Y$ be the graph of $f$.
Switching the roles of $X$ and $Y$, Theorem~\ref{thm:intpack}
yields
$
\upint\dboxm{t}\left(f^{-1}(y)\right)\dd\upack{s}(y)\leq\upack{s+t}(E).
$
Since the mapping $x\mapsto(x,f(x))$ is $c$-Lipschitz,
Lemma~\ref{lipp} yields $\upack{s+t}(E)\leq c^{s+t}\upack{s+t}(X)$.
The second inequality is proved the same way.
\end{proof}

\begin{rem}
All of the inequalities of this section remain true if all $\pack{}$'s
are replaced with $\boxm{}$'s, with the same proofs, one only has to use
uniform packings in place of packings.
In particular, Theorem~\ref{thm:lpack} reads
\begin{thm}\label{thm:box}
Let $X,Y$ be metric spaces. For any set $E\subs X\times Y$
\begin{enum}
\item  $\displaystyle\uboxm{gh}(E)
    \geq\upint\dboxm{g}(E_x)\dd\uboxm{h}(x)$,
\item  $\displaystyle\lboxm{gh}(E)
    \geq\upint\lboxm{g}(E_x)\dd\lboxm{h}(x)$,
\item  $\displaystyle\dboxm{gh}(E)
    \geq\upint\dboxm{g}(E_x)\dd\lboxm{h}(x)$,
\item  $\displaystyle\dboxm{gh}(E)
    \geq\inf_{x\in X}\dboxm{g}(E_x)\cdot\dboxm{h}(X)$.
\end{enum}
\end{thm}
\end{rem}

\section{Packing dimensions on cartesian products}
\label{sec:dim}

In this section we interpret the inequalities of the previous section
in terms of fractal dimensions. We first recall the dimensions and introduce
a new one related to the pre-measures $\dboxm s$ and $\dpack s$.
General reference:~\cite{MR1333890}.

Fix $E\subs X$. A family $\mathcal C$ of sets is a
\emph{$\del$-cover of $E$} if
it covers $E$ and $\diam C\leqslant\del$ for each $C\in\mathcal C$.
In this section we shall make frequent use of the
\emph{covering number function}
\begin{equation}\label{NNN}
  N_\del(E)=\min\{\card{\mathcal C}:\mathcal C
  \text{ is a $\del$-cover of $E$}\}, \quad \del>0.
\end{equation}

The well-known \emph{lower} and \emph{upper box dimensions},
(also called \emph{box-counting} or \emph{Minkowski})
of a nonempty set $E\subs X$ are equivalently defined, respectively, by
\begin{align*}
  \lbdim E & =
    \liminf_{\del\to 0}\frac{\log N_\del(E)}{\abs{\log r}}=
    \liminf_{\del\to 0}\frac{\log C_\del(E)}{\abs{\log r}},\\
  \ubdim E  &=
    \limsup_{\del\to 0}\frac{\log N_\del(E)}{\abs{\log r}}=
    \limsup_{\del\to 0}\frac{\log C_\del(E)}{\abs{\log r}}.
\end{align*}
Since
\begin{equation}\label{CxN}
  N_{2\del}(E)\leq C_\del(E)\leq N_\del(E)
\end{equation}
for any set $E$, the limits in these definitions indeed equal.
The upper and lower packing dimensions are, respectively, defined by,
cf.~\cite{MR1333890},
\begin{align*}
  \updim E&=\inf\{\sup_n\ubdim E_n:E\subs\bigcup_n E_n\},\\
  \lpdim E&=\inf\{\sup_n\lbdim E_n:E\subs\bigcup_n E_n\}.
\end{align*}
It is easy to check that the upper packing dimension may be equivalently defined
by $\updim E=\inf\{\sup_n\ubdim E_n:E_n\upto E\}$. However, this modification
of lower packing dimension gives a rise to a new dimension:
\begin{defn}\label{def:dpdim}
$\dpdim E=\inf\{\sup_n\lbdim E_n:E_n\upto E\}$
\end{defn}
It is clear that $ \lpdim X\leq \dpdim X\leq \lbdim X$ for any set $X$.
The following example shows that the three dimensions are distinct:
There is a compact set $X\subs\Rset$ such that $\lpdim X<\dpdim X<\lbdim X$:
\begin{ex}\label{ex:3x}
We will define three sets compact $K_0,K_1,E\subs\Rset$ such that
\begin{enum}
  \item $\lbdim K_0=\lbdim K_1=0$,
  \item $\dpdim K_0\cup K_1=\frac12$,
  \item $\lbdim E=1$,
  \item $E$ is countable.
\end{enum}
The required set is $X=K_0\cup K_1\cup E$. Indeed, (i) and (iv) imply
$\lpdim X=0$, (ii) and (iv) imply $\dpdim X=\frac12$ and (iii) implies $\lbdim X=1$.

To define the sets $K_0$ and $K_1$ consider the set $\ctree$ of all binary sequences
and also
the corresponding tree $\tree$ of finite binary sequences, and the canonical
mapping of $\ctree$ onto $[0,1]$ given by $\widehat x=\sum_{n\in\nset}2^{-n-1}x(n)$.
For $p\in\tree$ let $[p]=\{x\in\ctree:p\subs x\}$ be the cone determined by $p$.
It is clear that $C_p=\{\widehat x:x\in[p]\}$ is a closed binary interval of length
$2^{-\abs p}$.

Choose an infinite set $D\subs\nset$ such that
\begin{equation}\label{eq:limits}
\varliminf_{n\to\infty}\frac{\card{D\cap n}}{n}=0 \text{ and }
\varlimsup_{n\to\infty}\frac{\card{D\cap n}}{n}=1
\end{equation}
and set
\begin{align*}
  K_0&=\{\widehat x:x(n)=0\text{ for all $n\in D$}\},\\
  K_1&=\{\widehat x:x(n)=0\text{ for all $n\notin D$}\}.
\end{align*}
Let $n\in\nset$. There are exactly $2^{\abs{n\setminus D}}$ binary intervals
of length $2^{-n}$ that meet $K_0$. Hence
\begin{equation}\label{eq:limits2}
N_{2^{-n}}(K_0)=2^{\abs{n\setminus D}}
\end{equation}
and likewise
\begin{equation}\label{eq:limits3}
N_{2^{-n}}(K_1)=2^{\abs{n\cap D}}.
\end{equation}
Therefore~\eqref{eq:limits} yields
$$
  \lbdim K_0=
  \varliminf_{n\to\infty}\frac{\log N_{2^{-n}}(K_0)}{\abs{\log2^{-n}}}=
  \varliminf_{n\to\infty}\frac{\abs{n\setminus D}}{n}=
  1-\varlimsup_{n\to\infty}\frac{\abs{n\cap D}}{n}=0
$$
and likewise
$$
  \lbdim K_1=
  \varliminf_{n\to\infty}\frac{\log N_{2^{-n}}(K_1)}{\abs{\log2^{-n}}}=
  \varliminf_{n\to\infty}\frac{\abs{n\cap D}}{n}=0.
$$
Thus the sets $K_0,K_1$ satisfy (i). To show (ii), we first
claim that there is an infinite set $F\subs\nset$ such that
\begin{equation}\label{eq:limits999}
  \tfrac{n-1}{2}<\abs{n\cap D}\leq\tfrac n2.
\end{equation}
Indeed, \eqref{eq:limits} yields $n$ arbitrarily large such that
$\frac{\abs{D\cap n}}{n}\leq\frac12$ but $\frac{\abs{D\cap (n+1)}}{n+1}>\frac12$. Hence
$\abs{D\cap (n+1)}=\abs{D\cap n}+1$ and the two inequalities imply~\eqref{eq:limits999}.

Using~\eqref{eq:limits2}, \eqref{eq:limits3} and~\eqref{eq:limits999} it follows that
\begin{align*}
  \lbdim K_0\cup K_1&\leq
  \varliminf_{n\to\infty}\frac{\log(2^{\abs{n\setminus D}}+2^{\abs{n\setminus D}})}
  {\abs{\log2^{-n}}}\\
  &\leq
  \varliminf_{n\in F}\frac{\log(2^{(n+1)/2}+2^{n/2})}{n}\leq
  \varliminf_{n\in F}\frac{\log2^{1+(n+1)/2}}{n}=\frac12
\end{align*}
and in particular $\dpdim K_0\cup K_1\leq\frac12$.

To prove the opposite inequality suppose for contrary that $X_n\upto K_0\cup K_1$
are such that $\lbdim X_n<\frac12$ for all $n$.
With no harm done we may suppose $X_n$'s closed.
By the Baire category argument there is an open set $U$ that meets both
$K_0$ and $K_1$ and $\lbdim U\cap(K_0\cup K_1)<\frac12$.
Suppose without loss of generality that $U=I_0\cup I_1$, where $I_0$ meets $K_0$,
$I_1$ meets $K_1$ and $I_0$ and $I_1$ are non-overlapping binary intervals
of the same length, say $2^{-m}$.
If $n>m$, then the number of binary intervals of length $2^{-n}$ that meet
$I_0\cap K_0$ ($I_1\cap K_1$, respectively) is exactly $2^{\abs{A_n}}$
($2^{\abs{B_n}}$), where
$A_n=\{i\in\nset\setminus D:m\leq i<n\}$ and $B_n=\{i\in\nset\cap D:m\leq i<n\}$.
Since $\abs{A_n\cup B_n}=n-m$, we have $\max(\abs{A_n},\abs{B_n})\geq\frac{n-m}{2}$.
Thus $N_{2^{-n}}((K_0\cup K_1)\cap U)\geq2^{(n-m)/2}$, which in turn yields
$\lbdim(K_0\cup K_1)\cap U\geq\frac12$: the desired contradiction. We conclude that
$\dpdim K_0\cup K_1\geq\frac12$.
Thus (ii) holds.

Finally let $E=\bigl\{\frac{1}{\log n}:n\in\nset,n\geq2\bigr\}\cup\{0\}$.
Routine calculation proves that $\lbdim E=1$. Thus (iii) and (iv) hold.
\qed
\end{ex}
Let us now see how the dimensions we described are related to the measures and
pre-measures defined in the previous section.
It is easy to check that all of the measures
$\upack{s},\uboxm{s},\lpack{s},\lboxm{s}$ and the pre-measures
$\upack{s}_0,\uboxm{s}_0,\lpack{s}_0,\lboxm{s}_0,\dpack{s},\dboxm{s}$
are ``rarefaction indices'': If $\mathscr{L}^s$ is any of them, then
$$
  \inf\{s:\mathscr{L}^s(E)=0\}=\sup\{s:\mathscr{L}^s(E)=\infty\}.
$$
Each of these (pre)-measures is linked to one of the above fractal dimensions
by a common pattern:
Tricot~\cite{MR633256} proved that
$\ubdim E=\inf\{s:\upack{s}_0(E)=0\}=\inf\{s:\uboxm{s}_0(E)=0\}$ and also
that $\updim E=\inf\{s:\upack{s}(E)=0\}=\inf\{s:\uboxm{s}(E)=0\}$.
It is folklore (and very easy to prove) that
$\lbdim E=\inf\{s:\lboxm{s}_0(E)=0\}$ and
$\lpdim E=\inf\{s:\lboxm{s}(E)=0\}$, cf.~e.g.~\cite{MR1333890}. Combining with
Propositions~\ref{lem:00}---\ref{comp4} yields a list of equivalent definitions
of the dimensions under consideration.
\begin{prop}\label{dim:dp}
For any set $E\subs X$
\begin{enum}
\item $\ubdim E=\inf\{s:\upack{s}_0(E)=0\}=\inf\{s:\uboxm{s}_0(E)=0\}$,
\item $\lbdim E=\inf\{s:\lpack{s}_0(E)=0\}=\inf\{s:\lboxm{s}_0(E)=0\}$,
\item $\updim E=\inf\{s:\upack{s}(E)=0\}=\inf\{s:\uboxm{s}(E)=0\}$,
\item $\lpdim E=\inf\{s:\lpack{s}(E)=0\}=\inf\{s:\lboxm{s}(E)=0\}$,
\item
  $\dpdim E=\inf\{s:\dpack{s}(E)=0\}=\inf\{s:\dboxm{s}(E)=0\}\\
  =\inf\{s:\exists\Del\ \pack{s}_\Del(E)=0\}=
  \inf\{s:\exists\Del\ \boxm{s}_\Del(E)=0\}$.
\end{enum}
\end{prop}

Straightforward application of these identities to
Theorems~\ref{thm:box}, \ref{thm:lpack} and Corollary~\ref{int:rec2} yields the
corresponding dimension inequalities:
\begin{thm}\label{dim:sections}
Let $X,Y$ be metric spaces and $E\subs X\times Y$.
Let $A\subs X$ be a set such that $E_x\neq\emptyset$ for all $x\in A$. Then
\begin{enum}
\item  $\updim E\geq\updim A+\inf_{x\in A}\dpdim E_x$,
\item  $\lpdim E\geq\lpdim A+\inf_{x\in A}\lpdim E_x$,
\item  $\dpdim E\geq\dpdim A+\inf_{x\in A}\dpdim E_x$.
\end{enum}
\end{thm}
\begin{coro}\label{dim:rect}
For any metric spaces $X,Y$
\begin{enum}
\item  $\updim X\times Y\geq\updim X+\dpdim Y$,
\item  $\dpdim X\times Y\geq\dpdim X+\dpdim Y$,
\item  $\lpdim X\times Y\geq\lpdim X+\lpdim Y$.
\end{enum}
\end{coro}

\section{Solution of the Hu--Taylor problem}\label{sec:Xiao}

The Hu and Taylor~\cite{MR1376540} definition of $\aDim$ (cf.~\eqref{aDim0})
trivially extends to subsets of Euclidean spaces: for $X\subs\Rset^m$ let
$$
  \aDim X=\min\{\updim X\times Z-\updim Z: Z\subs\Rset^m\}.
$$
We employ the idea of Xiao~\cite{MR1388205} to show that
$\aDim X=\dpdim X$ for any $X\subs\Rset^m$ and actually for any
metric space of finite Assouad dimension.

We make heavy use of the capacity and covering number functions introduced in
\eqref{eq:capacity0} and~\eqref{NNN}. The following elementary estimates will be
needed. If $X,Y$ are metric spaces and $r>0$, then
\begin{align}
  C_r(X\times Y)&\leq N_r(X)\, C_r(Y), \label{prod1}\\
  N_r(X\times Y)&\leq N_r(X)\, N_r(Y). \label{prod2}
\end{align}

Let us recall the notion of Assouad dimension and related material.
The interested reader is referred to J.~Luukkainen's paper~\cite{MR1608518}.
Given $Q\geq0$ and $m\geq 0$, a metric space $(X,d)$ is termed \emph{$(Q,m)$-homogeneous}
if $\abs A\leq Q(b/a)^m$ whenever
$a>0$ and $b\geq a$ are numbers and $A\subs X$ a set with $a\leq d(x,y)\leq b$
if $x,y\in A$ and $x\neq y$. It is easy to check that $X$ is
$(Q,m)$-homogeneous if and only if $C_r(E)\leq Q\left(\frac{\diam E}{r}\right)^m$
for every set $E\subs X$ and every $r\leq\diam E$, and that is the definition we shall
use.

The space $X$ is termed \emph{$m$-homogeneous} if it is $(Q,m)$-homogeneous for some $Q$;
and $X$ is termed \emph{countably $(Q,m)$-homogeneous} if it is a countable union
of $(Q,m)$-homogeneous subspaces, and likewise
\emph{countably $m$-homogeneous} if it is a countable union
of $m$-homogeneous subspaces.

P.~Assouad~\cite{MR532401} defined what is now called \emph{Assouad dimension}:
If $X$ is a metric space, then
$$
  \adim X=\inf\{m>0:\text{$X$ is $m$-homogeneous}\}.
$$
We also introduce the countably stable modification of $\adim X$:
$$
  \sadim X=\inf\{\sup_i\adim X_i:\text{$\{X_i\}$ is a countable cover of $X$}\}.
$$

Spaces of finite Assouad dimension are also called \emph{$\beta$-spaces} or
\emph{doubling spaces} and various other names and similar concepts are in
use, e.g.~D.~G.~Larman's~\cite{MR0203691}. J.~Luukkainen's
paper~\cite{MR1608518} is a good source of information including an ample list
of references.

We shall need the following simple lemma.
\begin{lem}\label{homo1}
Let $X$ be a $(Q,m)$-homogeneous metric space.
\begin{enum}
\item If $0<r<t$, then $C_r(X)r^m\leq 2^m Q \, C_t(X)t^m$,
\item $\uboxm{m}_0(X)\leq 2^m Q\, \lboxm{m}_0(X)$.
\end{enum}
\end{lem}
\begin{proof}
(i) Suppose $C_t(X)<\infty$ and let $E\subs X$ be a maximal set with $\gap E>t$.
Then the family of balls $\{B(x,t):x\in E\}$ covers $X$. Therefore
$$
  C_r(X)\leq\sum_{x\in E}C_r(B(x,r))\leq\abs E\, Q\bigl(\tfrac{2t}{r}\bigr)^m
  \leq C_t(X)\,Q\,2^m\bigl(\tfrac tr\bigr)^m.
$$
(ii) Let $r_n\downarrow0$ be such that
$\lim C_{r_n}(X)r_n^m=\lboxm{m}_0(X)$.
If $r_{n+1}\leq r\leq r_n$, then (i) yields
$C_r(X)r^m\leq 2^mQ\,C_r(X)r_n^m$
and (ii) follows on letting $n\to\infty$.
\end{proof}

\begin{thm}\label{thm:HuTa}
Let $\Del$ be a scale and $0\leq s\leq m\in\nset$.
There is a compact set $\zz\subs\Rset^m$ such that
$\uboxm{m-s}(\zz)=\uboxm{m-s}_0(\zz)=1$
and
\begin{enum}
\item
$\uboxm{m}_0(X\times\zz)\leq2^mQ\,\boxm{s}_{\Del,0}(X)$
for every $(Q,m)$-homogeneous space $X$,
\item
$\uboxm{m}(X\times\zz)\leq2^mQ\,\boxm{s}_{\Del}(X)$
for every countably $(Q,m)$-homogeneous space $X$.
\end{enum}
\end{thm}
\begin{proof}
We prove only statement (i), as (ii) is its trivial consequence.
If $s=0$, put $\zz=[0,1]^m$. In this case the inequality reduces to
$\uboxm{m}_0(X\times[0,1]^m)\leq2^mQ\,\abs X$, which is trivially satisfied
for $X$ both finite or infinite.

If $s=m>0$, put $\zz=\{0\}$. In this case the inequality reduces to
$\uboxm{m}_0(X)\leq2^mQ\,\boxm{m}_{\Del,0}(X)$, which is nothing but
Lemma~\ref{homo1}.
We will thus suppose that $0<s<m$ and
let $p=\frac sm$ throughout the proof.

We first construct recursively a decreasing sequence $r_n\to0$ in $\Del$ and
an integer-valued sequence $g\in\nset^\nset$
such that, letting
\begin{align}
  G(n)&=g(0)g(1)\dots g(n-1), \notag \\
  u_n&=g(n)\,r_{n+1} \notag
\shortintertext{we have, for all $n$,}
   1-\tfrac1n&<G(n)r_n^{1-p}<1+\tfrac1n,\label{that3}\\
   r_{n+1}&<u_n<\tfrac{r_n}{n}. \label{that4}
\end{align}
Let $G(0)=1$, choose any $r_0\in\Del$ and recursively choose $r_{n+1}\in\Del$
small enough so that $r_{n+1}^p<r_n$ and the interval
$\Bigl(\frac{1-\frac{1}{n+1}}{G(n)r_{n+1}^{1-p}},
\frac{1+\frac{1}{n+1}}{G(n)r_{n+1}^{1-p}}\Bigr)$
is long enough to contain an even integer.
Let it be $g(n)$.
Thus
\begin{equation}\label{that6}
\frac{1-\frac{1}{n+1}}{G(n)r_{n+1}^{1-p}}<g(n)
<\frac{1+\frac{1}{n+1}}{G(n)r_{n+1}^{1-p}}
\end{equation}
and therefore $1-\frac{1}{n+1}<G(n+1)r_{n+1}^{1-p}<1+\frac{1}{n+1}$, as required.
Since $g(n)\geq2$, we also have $r_{n+1}<u_n$.
Condition $r_{n+1}^p<r_n$ in conjunction with~\eqref{that6} ensures $u_n<r_n/n$.

Next we define the space $\zz$.
Let
\begin{align*}
  &\TT=
  \{x\in\nset^\nset:\forall n\ x(n)<g(n)\},\\
  &\ttt=\{\tau\in\nset^{<\nset}:\exists x\in\TT\ \tau\subs x\},
  \quad\ttt_n=\{\tau\in\ttt:\abs\tau=n\}.
\end{align*}
For each $\tau\in\ttt$ define intervals $I_\tau=[a_\tau,b_\tau]$
recursively as follows. $I_\emptyset=[0,r_0]$.
Now suppose $n\in\nset$, $\tau\in\ttt_n$ and $I_\tau=[a_\tau,b_\tau]$
is defined.
For $i<g(n)$ let $a_{\tau i}=a_\tau+ir_{n+1}$,
$b_{\tau i}=a_{\tau i}+r_{n+1}=a_\tau+(i+1)r_{n+1}$.
It is clear that, for any $\tau$ of length $n$, the family $\{I_{\tau i}:i<g(n)\}$
consists of adjacent non-overlapping equally sized intervals of length $r_{n+1}$
and that its union is an interval of length $u_n$; and since $u_n<r_n$,
the union is contained in $I_\tau$.
Set
$$
  \yy=\bigcap_{n\in\nset}\bigcup\{I_\tau:\tau\in\ttt_n\},
  \quad
  \zz=\yy^m.
$$
In order to show that $\uboxm{m-s}(\zz)=\uboxm{m-s}_0(\zz)=1$
it is enough to establish the following claim.
Let $\mu$ be the evenly distributed Borel probability measure on $\zz$.
In more detail, $\mu$ is the cartesian power of the Borel measure $\lambda$
on $\yy$ that is determined by its values on $I_\tau$'s:
if $\tau\in\ttt_n$, then $\lambda(I_\tau)=1/G(n)$.
\begin{lem}\label{xi111}
$\mu(E)=\uboxm{m-s}(E)$ for every Borel set $E\subs\zz$
and $\uboxm{m-s}_0(\zz)=1$.
\end{lem}
\begin{proof}
For each $n\in\nset$ and every
$\ttau=\seq{\tau_i:i<m}\in(\ttt_n)^m$ define
$J_\ttau=\prod_{i<m}[a_{\tau_i},a_{\tau_i+u_n})$.
Let $B\subs\yy$ be Borel. Consider the set
$
  S=\{\ttau\in(\ttt_n)^m:J_\ttau\cap B\neq\emptyset\}.
$
Pick one point of $B$ in every $J_\ttau$, $\ttau\in S$. Thus chosen points are mutually
more than $r_n-u_n$ apart and thus witness $C_{r_n-u_n}(B)\geq\abs S$.
On the other hand, $\mu(B)\leq\sum_{I\in S}\mu(I)=\abs S\cdot\frac{1}{G(n)^m}$.
Therefore
$$
  \mu(B)\leq\frac{\abs S}{G(n)^m}
  \leq\frac{C_{r_n-u_n}(B)}{G(n)^m}
  \clue{that3}{\leq} C_{r_n-u_n}(B)r_n^{m-s}\frac{1}{(1-1/n)^m}
$$
and since $\tfrac{1}{(1-1/n)^m}\to1$ and, by \eqref{that4},
$\tfrac{u_n}{r_n}\to1$, we get
$$
  \mu(B)\leq\varlimsup_{n\to\infty}C_{r_n-u_n}(B)(r_n-u_n)^{m-s}
  \leq\uboxm{m-s}_0(B).
$$
Since this holds for every Borel set $B$, we get $\mu(E)\leq\uboxm{m-s}(E)$
for every $E\subs\zz$ Borel.

It remains to show that $\uboxm{m-s}_0(\zz)\leq1$.
Write $\teta_n=1+\frac1n$.
Since obviously $N_{u_n}(\yy)=G(n)$, we have, for all $r\in[u_n,r_n]$,
\begin{equation}\label{xi3}
  N_r(\yy)r^{1-p}\leq N_{u_n}(\yy)r_n^{1-p}=G(n)r_n^{1-p}
    \clue{that3}{\leq}\teta_n.
\end{equation}
Now suppose $r\in[r_{n+1},u_n]$.
Note first that $N_r(E)\leq \frac{\diam E}{r}+1$ for any set $E\subs\Rset$
and therefore $N_r(\yy)\leq N_t(\yy)\bigl(\frac tr+1\bigr)$ whenever $r<t$.
In particular,
\begin{equation}\label{new1}
  N_r(\yy)\leq N_{u_n}(\yy)\Bigl(\frac{u_n}{r}+1\Bigr)\leq2G(n)\frac{u_n}{r},
\end{equation}

\begin{equation} \label{xi5}
\begin{split}
    &N_r(\yy)r^{1-p}\leq G(n)(\tfrac{u_n}{r}+1\bigr)r^{1-p}
     =G(n+1)r_{n+1}r^{-p}+G(n)r^{1-p}\\
    &\qquad\leq G(n+1)r_{n+1}^{1-p}+G(n)r_n^{1-p}\bigl(\tfrac{u_n}{r_n}\bigr)^{1-p}
     \clue{that3}{\leq}
     \teta_{n+1}+\teta_n\bigl(\tfrac{u_n}{r_n}\bigr)^{1-p}.
  \end{split}
\end{equation}
Since $C_r(\zz)\leq (N_r(\yy))^m$ (cf.~\eqref{CxN} and~\eqref{prod2}) and
since $\teta_n\to1$ and $\frac{u_n}{r_n}\leq\frac1n$, the estimates~\eqref{xi3}
and~\eqref{xi5} give
$$
  \uboxm{m-s}_0(\zz)\leq\varlimsup_{r\to0}N_r(\zz)r^{m-s}\leq
  \varlimsup_{n\to\infty}(\teta_n+\teta_n/n^{1-p})^m=1.
  \qedhere
$$
\end{proof}
We proceed with the proof of the theorem.
It remains to estimate $\uboxm{m}_0(X\times\zz)$ from above.
For $r\in[u_n,r_n]$ we employ Lemma~\ref{homo1}(i).

\begin{equation} \label{xi4}
  \begin{split}
    &C_r(X\times\zz)r^m
    \leq C_r(X)r^m\bigl(N_r(K)\bigr)^m
    \leq 2^mQ\, C_{r_n}(X)\,r_n^mG(n)^m\\
    &\qquad\leq2^mQ\, C_{r_n}(X)\,r_n^s\bigl(G(n)r_n^{1-p}\bigr)^m
    \clue{that3}{\leq}2^mQ\, C_{r_n}(X)\,r_n^s\,\teta_n^m.
  \end{split}
\end{equation}
For $r\in[r_{n+1},u_n]$ we employ the latter estimate~\eqref{new1}.

\begin{equation} \label{xi6}
  \begin{split}
  &C_r(X\times\zz)r^m
    \leq C_r(X)r^m\bigl(N_r(K)\bigr)^m
    \leq C_{r_{n+1}}(X)r_{n+1}^s\bigl(2G(n)u_nr_{n+1}^{-p}\bigr)^m
       \\
    &\qquad
    \leq C_{r_{n+1}}(X)r_{n+1}^s\bigl(2G(n+1)r_{n+1}^{1-p}\bigr)^m
      \clue{that3}{\leq} 2^m\teta_{n+1}^mC_{r_{n+1}}(X)r_{n+1}^s.
  \end{split}
\end{equation}
Since $\teta_n\to1$, \eqref{xi4} and~\eqref{xi6} yield
$$
  \uboxm{m}_0(X\times\zz)
  \leq\varlimsup_{n\to\infty}2^mQ\,C_{r_n}(X)r_n^s
  \leq 2^mQ\,\boxm{s}_{\Del,0}(X).
  \qedhere
$$
\end{proof}

The simple proof of the following corollary is omitted.
\begin{coro}
Let $0\leq s\leq m\in\nset$. Let $X$ be a countably $m$-homogeneous metric space
such that $\dboxm{s}(X)<\infty$.
There is a compact set~$\zz\subs\Rset^m$ such that $\uboxm{m-s}(\zz)=1$
and $\uboxm{m}(X\times\zz)$ is $\sigma$-finite.
\end{coro}

\begin{thm}\label{HuTa2}
Let $X$ be a metric space $X$ and $\sadim X\leq m\in\nset$.
There is a compact set $Z\subs\Rset^m$ such
that $\updim X\times Z=\dpdim X+\updim Z$.
\end{thm}
\begin{proof}
Fix $s>\dpdim X$. By Proposition~\ref{dim:dp} there is a scale $\Del$ such that
$\boxm{s}_\Del(X)=0$. Let $\zz_s$ be the space $\zz$ of Theorem~\ref{thm:HuTa}.

Let $\eps>0$. There is a cover $X=\bigcup_{i\in\nset}X_i$ such that for each
$i\in\nset$ there is $Q_i$ such that $X_i$ is $(Q_i,m+\eps)$-homogeneous
and moreover $\boxm{s}_{\Del,0}(X_i)<1$.
Inspect the estimates~\eqref{xi4} and~\eqref{xi6}: If we use
$(Q_i,m+\eps)$-homogeneity of $X_i$ instead of $(Q,m)$-homogeneity of $X$, we get
for $r\in[u_n,r_n]$
$$
  C_r(X_i\times\zz_s)r^{m+\eps}\leq
  2^{m+\eps}Q_i\teta_n^m\,C_{r_n}(X_i)r_n^sr_n^\eps
$$
and since~\eqref{xi6} does not depend on homogeneity of $X$, we get
for $r\in[r_{n+1},u_n]$
$$
  C_r(X_i\times\zz_s)r^{m+\eps}\leq
  C_r(X_i\times\zz_s)r^mr^\eps\leq
  2^m\teta_{n+1}^m\,C_{r_{n+1}}(X_i)r_{n+1}^su_n^\eps.
$$
Since $\teta_n\to0$, $r_n\to0$, $u_n\to0$ and
$\varlimsup_{n\to\infty} C_{r_n}(X_i)r_n^s\leq\boxm{s}_{\Del,0}(X_i)<1$,
these es\-ti\-mates yield $\uboxm{m+\eps}_0(X_i\times\zz_s)=0$ for all $i$.
Consequently $\uboxm{m+\eps}(X\times\zz_s)=0$ for all $\eps>0$,
whence $\updim X\times\zz_s\leq m$.

Now let $s_k=\dpdim X+\frac1k$. Consider the spaces $\zz_{s_k}$.
\emph{Mutatis mutandis} we may assume $\zz_{s_k}\subs[2^{-k-1},2^{-k}]$.
Let $Z=\bigcup_{k\in\nset}\zz_{s_k}\cup\{0\}$. It is clearly compact and
$\updim X\times Z\leq\sup\updim X\times \zz_{s_k}\leq m$. On the other hand,
$\updim Z\geq\sup m-s_k=m-\dpdim X$.
Thus $\updim X\times Z\leq\dpdim X+\updim Z$.
The opposite inequality follows from Theorem~\ref{dim:rect}.
\end{proof}

In particular, if $X\subs\Rset^m$, the above theorem yields a solution
to the problem of Hu and Taylor:
\begin{coro}
For every set $X\subs\Rset^m$ there is a compact set $Z\subs\Rset^m$ such that
$\updim X\times\zz-\updim\zz=\dpdim X$. In particular,
$\aDim X=\dpdim X$.
\end{coro}

\section{Comments and questions}\label{sec:rem}

\subsection*{Other packing measures}
The packings we used are sometimes called \emph{weak packings} or
\emph{pseudo-packings}.
There are other kinds of packing in use. The most common that we shall call
\emph{true packing} is this: $\{(x_i,r_i):i\in I\}$
is a \emph{true packing} if the balls $B(x_i,r_i), B(x_j,r_j)$ are disjoint
for distinct $i,j\in I$. There are also open balls variants.
Some definitions of packing measures are based on diameters of the underlying balls,
instead of radii.
Various packing measures and their relations are discussed in detail
e.g.~in~\cite{MR1844397,MR2288081,MR1824158}.
Analysis shows that our results are to some extent
valid also for other packing measures.

\subsection*{Hausdorff measures}
We intentionally neglected results involving Hausdorff measure and dimension.
The reason is that for any Hausdorff function $h$ we have
$\haus{h_2}\leq\lboxm{h}$, where $h_2(r)=h(r/2)$ and $\haus{h_2}$ is the corresponding
Hausdorff measure. (Hint: If $\{(x_i,\delta)\}$ is a maximal uniform packing, then
$\{B(x_i,\delta)\}$ is a $2\delta$-cover.)
Thus e.g.~Howroyd's~\cite[Theorem 13]{MR1362951} stating that
$\upack{gh}(X\times Y)\geq\upack{g}(X)\,\haus{h_2}(Y)$
follows at once from Corollary~\ref{int:rec2}.

\subsection*{Upper estimates of cartesian products measures}
All inequalities of Section~\ref{sec:int} estimate the measures on a product by
means of measures on coordinate spaces from below. We paid no attention to
reverse estimates. Basic results in this direction are due to
Tricot~\cite{MR633256} and Howroyd~\cite{MR1362951}, see also~\cite{MR1844397}.
Howroyd has the following: Let $\uPack g$ denote the upper packing measure
obtained from true packings. Let $g,h$ be right-continuous Hausdorff functions. Then
$\uPack{gh}(X\times Y)\leq\uPack g(X)\upack h(Y)$ for any metric spaces $X,Y$, as long
as the product on the right is not $0\cdot\infty$ or $\infty\cdot0$.
Inspection of the proof shows that restricting the admissible radii to a given scale
does not matter. One can thus conclude that, under the same conditions and with
the obvious definitions,
$\uPack{gh}_\Del(X\times Y)\leq\uPack g_\Del(X)\upack h_\Del(Y)$ for any scale $\Del$
and also $\lPack{gh}(X\times Y)\leq\lPack g(X)\upack h(Y)$ and
$\dPack{gh}(X\times Y)\leq\dPack g(X)\upack h(Y)$.

Corresponding inequalities for the box measures can be derived from~\eqref{prod1}.
Corresponding inequalities for dimensions (due to Tricot~\cite{MR633256})
are well-known (except $\dpdim X\times Y\leq\dpdim X+\updim Y$).
Combining Tricot's inequalities with the one just mentioned and
Corollary~\ref{dim:rect} we thus have:
\begin{thm}
For any metric spaces $X,Y$
\begin{align*}
  \lpdim X+\lpdim Y\leq\lpdim X\times Y&\leq\lpdim X+\updim Y, \\
  \dpdim X+\dpdim Y\leq\dpdim X\times Y&\leq\dpdim X+\updim Y \\
  &\leq\updim X\times Y\leq\updim X+\updim Y.
\end{align*}
\end{thm}

\subsection*{Comparison of lower packing and box measures}
As is obvious from Section~\ref{sec:measures}, the measures $\lpack h$ and $\lboxm h$
are closely related, much closer than their upper counterparts,
but we do not really know much about their relation.
We even do not know if they are equal. The following problems seem interesting.
\begin{question}
Is there a (compact) set $X\subs\Rset$ and $s>0$ such that
\begin{enum}
\item
$\lboxm s(X)<\lpack s(X)$?
\item
$\lboxm s(X)=0$ and $\lpack s(X)=\infty$?
\item
$\lboxm s(X)=0$ and $0<\lpack s(X)<\infty$?
\end{enum}
\end{question}

A related problem, perhaps the most interesting one, is whether one can replace
$\lboxm g$ with $\lpack g$ in the integrands in Theorem~\ref{thm:lpack}:

\begin{question}
Is it true that inequalities in Theorem~\ref{thm:lpack} improve to
$\upack{gh}(E)\geq\upint\dpack{g}(E_x)\dd\upack{h}$,
$\lpack{gh}(E)\geq\upint\lpack{g}(E_x)\dd\lpack{h}$ and
$\dpack{gh}(E)\geq\upint\dpack{g}(E_x)\dd\lpack{h}$?
\end{question}
A modest variation of this problem:
\begin{question}
Let $X,Y$ be metric spaces and $g,h$ Hausdorff functions.
\begin{enum}
\item
Suppose $\dpack{g}(X)>0$ and $\upack{h}(Y)>0$.
Does it follow that $\upack{gh}(X\times Y)>0$?
\item
Suppose $\lpack{g}(X)>0$ and $\lpack{h}(Y)>0$.
Does it follow that $\lpack{gh}(X\times Y)>0$?
\end{enum}
\end{question}

Another interesting problem is that of semifiniteness of the $\lpack h$ and $\lboxm h$.
Recall that a Borel measure is semifinite if every Borel set of infinite measure
contains a Borel subset of finite positive measure.
By a theorem of H.~Joyce and D.~Preiss~\cite{MR1346667} the upper packing measure
$\upack h$ on an analytic metric space is semifinite.
\begin{question}
\begin{enum}
\item Under what conditions imposed on $X$ and $g$ are the measures $\lpack g$
and $\lboxm g$ on $X$ semifinite?
\item Is there $s>0$ such that $\lpack s$ is not semifinite on $\Rset$?
\end{enum}
\end{question}

\subsection*{Directed pre-measures}
We also do not know if $\dboxm g$ and $\dpack g$ may differ.
\begin{question}
Is there a (compact) set $X\subs\Rset$ and $s>0$ such that
\begin{enum}
\item
$\dboxm s(X)<\dpack s(X)$?
\item
$\dboxm s(X)=0$ and $\dpack s(X)=\infty$?
\item
$\dboxm s(X)=0$ and $0<\dpack s(X)<\infty$?
\end{enum}
\end{question}

\subsection*{Finite Assouad dimension hypothesis}
Proofs of Section~\ref{sec:Xiao} inevitably depend on the finite
Assouad dimension of the metric space under consideration, but there is no clue
that this hypothesis is not superfluous.
\begin{question}
Is there a metric space $X$ such that $\updim X<\infty$ and
$$
  \dpdim X<\inf\{\updim X\times Z-\updim Z:\updim Z<\infty\}?
$$
\end{question}

\subsection*{Dimensions of Borel measures}
Our dimension inequalities have counterparts for dimensions of finite Borel measures.
Recall that if
$\mu$ is a finite Borel measure in a metric space $X$ and $\dim$ is any of the
fractal dimension under consideration, the corresponding dimensions of $\mu$ are
defined by
$$
  \dim\mu=\inf\{\dim E:B\subs X\text{ Borel, } \mu(E)>0\}.
$$
It is easy to check that $\lpdim\mu=\dpdim\mu=\lbdim\mu$ and $\updim\mu=\ubdim\mu$.
Another equivalent definition of the two dimensions is
$\lpdim\mu=\sup\{s:\mu\ll\lpack s\}=\sup\{s:\mu\ll\lboxm s\}$ and
$\updim\mu=\sup\{s:\mu\ll\upack s\}=\sup\{s:\mu\ll\uboxm s\}$,
where $\ll$ denotes absolute continuity.

The following theorem is a straightforward consequence of Corollary~\ref{dim:rect}.
\begin{thm}
Let $\mu,\nu$ be finite Borel measures in metric spaces.
\begin{enum}
\item  $\updim\mu\times\nu\geq\updim\mu+\lpdim\nu$,
\item $\lpdim\mu\times\nu\geq\lpdim\mu+\lpdim\nu$.
\end{enum}
\end{thm}
There is also a measure counterpart to Theorem~\ref{HuTa2}.
\begin{thm}
Let $\mu$ be a finite Borel measure in a metric space $X$.
If $\sadim X\leq m\in\nset$,
then
$$
  \inf\{\updim\mu\times\nu-\updim\nu:\text{$\nu$ is a finite Borel measure in }\Rset^m\}
  =\lpdim\mu.
$$
\end{thm}
\begin{proof}[Proof in outline]
Let $\eps>0$ and $s=\lpdim\mu+\eps$. There is a set $E\subs X$ such that
$\mu(E)>0$ and $\dpdim E<s$. By Theorem~\ref{HuTa2} and its proof there is
a compact set $\zz_s\subs\Rset^m$ such that $\updim E\times\zz_s\leq m$
and, by Lemma~\ref{xi111}, the corresponding measure $\mu_s$ on $\zz_s$
satisfies $\updim\mu_s\geq m-s$. Thus
$\updim\mu\times\mu_s\leq\updim E\times\zz_s\leq m\leq
\updim\mu_s+\lpdim\mu+\eps$.
\end{proof}

\bibliographystyle{amsplain}
\providecommand{\bysame}{\leavevmode\hbox to3em{\hrulefill}\thinspace}
\providecommand{\MR}{\relax\ifhmode\unskip\space\fi MR }
\providecommand{\MRhref}[2]{%
  \href{http://www.ams.org/mathscinet-getitem?mr=#1}{#2}
}
\providecommand{\href}[2]{#2}

\end{document}